\documentclass[12pt,oneside]{amsart}
\usepackage[english]{babel}
\usepackage{amssymb}
\usepackage{amsthm}
\usepackage{mathtools}

\usepackage[foot]{amsaddr}
\usepackage{amsmath}
\usepackage{a4wide}
\usepackage{nccmath}
\usepackage{esvect}
\usepackage{ulem}
\usepackage{enumitem}
\usepackage{hyperref}
\usepackage{cleveref}
\usepackage{subfiles}
\hypersetup{
    colorlinks,
    linkcolor=red,
    citecolor=black,
    urlcolor=blue
}
\usepackage{verbatim}
\usepackage[dvips]{graphicx}
\usepackage{t1enc}
\usepackage{color}
\usepackage[margin=0.5in]{geometry}
\usepackage{graphicx}
\usepackage{color,soul}
\usepackage[parfill]{parskip}   
\usepackage[numbers,sort&compress]{natbib}
\usepackage{tikz}
\usepackage{amsmath}
\usepackage{caption}
\usepackage{subcaption}
\usepackage{mathtools}
\usepackage[section]{placeins}

\usetikzlibrary{arrows,positioning}
\usetikzlibrary{graphs}
\usetikzlibrary{graphs.standard}

\makeatletter
\def\subsubsection{\@startsection{subsubsection}{3}%
  \z@{.5\linespacing\@plus.7\linespacing}{.1\linespacing}%
  {\normalfont\itshape}}
\makeatother

\newtheorem{thm}{Theorem}[section]

\newtheorem{conjecture}[thm]{Conjecture}

\newtheorem{question}[thm]{Question}

\newcommand{\ignore}[1]{}

\tikzset{
    dot/.style 2 args={fill, circle, inner sep=1pt, %label={#1:\scriptsize #2}
    }
}

\begin{document}

\author{Rebekah Herrman}
\address[Rebekah Herrman]{Department of Industrial and Systems Engineering, The University of Tennessee Knoxville, Knoxville, TN}
\email[Rebekah Herrman]{rherrma2@utk.edu}

\author{Grace Wisdom}
\address[Grace Wisdom]{Department of Mathematics, The University of Tennessee Knoxville, Knoxville, TN}
\email[Grace Wisdom]{ewisdom@vols.utk.edu}

\title[Leaky forcing and resilience of Cartesian products of $K_n$] 
{Leaky forcing and resilience of Cartesian products of $K_n$}

\linespread{1.3}
\pagestyle{plain}

\begin{abstract}
Zero forcing is a process on a graph $G = (V,E)$ in which a set of initially colored vertices,$B_0(G) \subset V(G)$, can color their neighbors according to the color change rule. The color change rule states that if a vertex $v$ can color a neighbor $u$ if $u$ is the only uncolored neighbor of $v$. If a vertex $v$ colors its neighbor, $u$, $v$ is said to force $u$. Leaky forcing is a recently introduced variant of zero forcing in which some vertices cannot force their neighbors, even if they satisfy the color change rule. This variation has been studied for limited families of graphs with particular structure, such as products of paths and discrete hypercubes. A concept closely related to $\ell$-leaky forcing is $\ell$-resilience. A graph is said to be $\ell$-resilient if its $\ell$-leaky forcing number equals its zero forcing number. In this paper, we prove direct products of $K_n$ with $P_t$ and $K_n$ with $C_t$ is 1-resilient and conjecture the former is not 2-resilient.
\end{abstract}
\maketitle 



\section{Introduction}

%\rh{first describe zero forcing and applications. then describe motivation for leaky forcing and define it. mention and cite result that $Z_\ell(G) \leq Z_0(G)$ for all $\ell \geq 1$ where $Z_0(G)$ denotes zero forcing number. mention leaky forcing application paper \cite{abbas2023resilient}.}

%\rh{then briefly state main results of paper and give road map for paper structure.}

Zero forcing is a vertex coloring process that has origins in quantum control theory \cite{burgarth2007full} and applications to numerous fields such as network controllability \cite{monshizadeh2014zero, trefois2015zero, mousavi2018null}. In the zero forcing process on a graph $G = (V,E)$, a subset of vertices $B \subset V(G)$ is initially colored blue. A blue vertex $v$ can color an uncolored neighbor $u$, or $v$ can \textit{force} $u$, if $u$ is the only uncolored neighbor of $v$. This is known as the \textit{color change rule}. If the set $B$ can eventually force all vertices of $G$, we say that $B$ is a \textit{zero forcing set}. The zero forcing number of $G$ is denoted $Z(G)$ and is equal to the size of the smallest cardinality zero forcing set. Previous work in zero forcing includes determining how long it takes for a zero forcing set to infect a graph \cite{hogben2012propagation, butler2013throttling}, degree-based bounds for the zero forcing number \cite{gentner2018some, davila2018lower}, and approximation algorithms and machine learning models for finding zero forcing numbers and sets \cite{cameron2024approximation, ahmad2023graph}.

In 2019, Dillman, and Kenter introduced $\ell$-leaky forcing, a variation of zero forcing in which $\ell$ vertices do not enforce the color change rule \cite{dillman2019leaky}. $Z_{(\ell)}(G)$ is used to denote the $\ell$-leaky forcing number, which is the size of a smallest cardinality set that can force $G$ even with the addition of $\ell$ leaks. We shall call the set $\ell$-leaky forcing vertices $B_{\ell}$. Zero forcing is equivalent to $\ell$-leaky forcing when $\ell=0$, thus, throughout this work we will use $Z_0(G)$ to denote the zero forcing number of $G$ and $B_0$ to denote a minimum zero forcing set. Recent $\ell$-leaky forcing work includes exactly determining the $\ell$-leaky forcing number (or proving upper bounds) for some families of graphs \cite{herrman2022d}, different variations of $\ell$-leaky forcing \cite{elias2023leaky, alameda2020generalizations}, applications in computer science \cite{abbas2023resilient}, and the introduction of \textit{$\ell$-resilience} \cite{alameda2022leaky}. A graph $G$ is said to be $\ell$-resilient if its $\ell$-leaky forcing number is equal to its zero forcing number. 

Products of graphs have gained significant attention in zero forcing \cite{huang2010minimum, karst2020blocking, cameron2023forts, brevsar2017grundy, leclair2024zero} and other graph infection processes such as bootstrap percolation \cite{brevsar2024bootstrap, gao2015bootstrap, gravner2017bootstrap, hedvzet20233}. In this work, we first calculate the $1$-leaky forcing number of direct products of $K_n$ with $P_t$, $C_t$, and itself in Sec.~\ref{sec:directprod}. The \textit{direct product} of two graphs $G$ and $H$, denoted $G \times H$, is a graph with vertex set $V(G) \times V(H)$. Two vertices $(g,h)$ and $(g',h')$ form an edge in $G \times H$ if and only if $gg' \in E(G)$ and $hh' \in E(H)$. In the same section, we conjecture that $K_n \times P_t$ is not $2$-resilient.

%Then, in Sec.~\ref{sec:cartesianprod}, we consider Cartesian products of graphs $G$ and $H$. The \textit{Cartesian product} of $G$ and $H$, written $G \boxtimes H$ is a graph with vertex set $V(G) \times V(H)$. Vertices $(u,v)$ and $(u'v')$ form an edge in $G \boxtimes H$ if $u=u'$ and $vv' \in E(H)$ or $v=v'$ and $uu' \in E(G)$. 

In Sec.~\ref{sec:conclusion}, we summarize the main results of this work and introduce related open questions.

\section{$1$-resilience of direct products including $K_n$}\label{sec:directprod}
We first consider the $1$-leaky forcing number of paths on $t$ vertices, $P_t$, with complete graphs on $n$ vertices, $K_n$. 
\begin{thm}\label{thm:pathcomplete}
   $K_n \times P_t$ is $1$-resilient. Specifically,
        \[Z_1(K_n \times P_t) = Z_0(K_n \times P_t) = \begin{cases} 
           (n-2)t & t \; \mathrm{even}, \; n \geq 3 \\
           (n-2)t + 2 & t \; \mathrm{odd}, \; n\geq 2
            \end{cases}.
    \]
\end{thm}
%\rh{can we show it is not 2-resilient?}

\begin{proof}
Draw $K_n \times P_t$ such that vertices are arranged in $n$ rows and $t$ columns. Label vertices such that the top left is $(1,1)$, bottom left is $(n,1)$, the top right is $(1,t)$ and the bottom right is $(n,t)$. In general, the first coordinate of each vertex label refers to the row and the second coordinate refers to the column.
Benson, Ferrero, Flagg, Furst, Hogben, Vasilevska, and Wissman \cite{benson2018zero} showed that
\[Z_0(K_n \times P_t) = \begin{cases} 
           (n-2)t & t \; \mathrm{even}, \; n\geq 3 \\
           (n-2)t + 2 & t \; \mathrm{odd}, n \geq 2
            \end{cases}.
    \]
 Since $Z_{(0)}(G) \leq Z_{(1)}(G),$ we need only construct sets of the appropriate sizes to show the result.

\noindent \textbf{Case 1: $t$ even.}
 We consider two subcases: $n$ even and $n$ odd.

\noindent \textbf{Subcase 1.1: n even.} 

Let the initially infected set be vertices $\{(i,1)\}_{i =3}^n \cup \{(j, t)\}_{j=1}^{n-2}  \cup \{(f, 2)\}_{f \notin \{1,n\}} \cup \{ (m, 2k) \}_{k \notin \{1, \frac{t}{2}\}, m \notin \{\frac{n}{2}, \frac{n}{2}+1\}} \cup \{(q,2p+1 )\}_{p \neq 0, q \notin \{1,n\}}$. Note that this set has size $(n-2)t$ as each of the $t$ columns contains exactly $n - 2$ initially infected vertices. See Fig.~\ref{fig:Theorem-2.1.1} for an example of an initially infected set.

There are two subcases to consider: there is a leak in column $1$ or $t$, or there is a leak in what we will call an ``interior column", that is a column that is not column $1$ or $t$.

\textbf{Subcase 1.1.1: Leak in column $1$ or $t$.}

Suppose there is a leak in column $1$, without loss of generality. Then the forcing process starts from column $t$. The vertex $(1,t)$ forces $(n,t-1)$, and once this is forced, $(2,t)$ forces $(1, t-1)$. Note then that $(1, t-2)$ forces $(n, t-3)$ and $(n, t-2)$ forces $(1, t-3)$. Note that since there are no leaks until the first column, this process iterates until column 1 is forced, with even indexed columns $2p$ forcing odd indexed columns $2p-1$. Once all odd columns have been forced, note that there is at most $1$ leak in column 1. This means either $(1,1)$ can force $(n,2)$ or $(1,n)$ can force $(1,2)$, or both. Once either $(n,2)$ or $(1,2)$ is forced, the other is forced by any vertex $(1,j)$ for $j \in \{2, 3, \ldots , n-1\}$. Since columns 1 through 3 have all been forced, column 3 can then force column 4, and odd indexed columns $2q+1$ force even indexed columns $2(q+1)$ until all vertices have been forced. If there is instead a leak in column $t$, a similar process occurs where odd-indexed columns force even-indexed columns until column $t$ is forced, and then even-indexed columns force odd-indexed ones until the entire graph is forced.

\textbf{Subcase 1.1.2: Leak in an interior column $h$.}

Suppose there is a leak in some interior column. Without loss of generality, suppose the column with a leak occurs in the first $\frac{t}{2}$ columns at index $h$. As in the previous case, odd-indexed columns will force even-indexed columns (from the direction of column $1$), and even-indexed columns will force odd indexed ones (from the direction of column $t$), until column $h$ is reached. If $h$ is odd, $h$ is forced by $h+1$ and column $h-1$ was forced by column $h-2$. If $h$ is even, $h$ was forced by column $h-1$ and column $h+1$ was forced by column $h+2$. In both cases, two columns that are next to each other have been completely forced. In the former case, column $h-1$ can force $h-2$, allowing column $h-3$ to force $h-4$ and so on until all vertices in columns left of $h$ are forced. Similarly, column $h$ can force at least one vertex in column $h+1$. Since column $h$ has at most one leak, when one of the uncolored vertices of column $h+1$ is forced, the other is then also forced by a vertex in $h$, as every vertex in $h$ now has at most one uncolored neighbor. This allows column $h+2$ to force $h+3$ and $h+4$ to force $h+5$ until column $t$ is forced. Thus, all veretices are forced. The case when $h$ is even is analogous.

\begin{figure}
    \centering
    \begin{tikzpicture} [scale=2.75]
    \begin{scope} [every node/.style={scale=.75,circle,draw}]
    \begin{scope} [every node/.style={scale=.75,circle,draw, fill=blue}]
    
        \node (A) at (0, 0) {};
        \node (B) at (0, 1) {};
        \node (F) at (1, 1) {};
        \node (G) at (1, 2) {};
        \node (J) at (2, 1) {};
        \node (K) at (2, 2) {};
        \node (P) at (3, 3) {};
        \node (M) at (3, 0) {};
        \node (R) at (4, 1) {};
        \node (S) at (4, 2) {};
        \node (W) at (5, 2) {};
        \node (X) at (5, 3) {};
        
    \end{scope}
    
        \node (A) at (0, 0) {};
        \node (B) at (0, 1) {};
        \node (C) at (0, 2) {};
        \node (D) at (0, 3) {};

        \node (E) at (1, 0) {};
        \node (F) at (1, 1) {};
        \node (G) at (1, 2) {};
        \node (H) at (1, 3) {};

        \node (I) at (2, 0) {};
        \node (J) at (2, 1) {};
        \node (K) at (2, 2) {};
        \node (L) at (2, 3) {};

        \node (M) at (3, 0) {};
        \node (N) at (3, 1) {};
        \node (O) at (3, 2) {};
        \node (P) at (3, 3) {};

        \node (Q) at (4, 0) {};
        \node (R) at (4, 1) {};
        \node (S) at (4, 2) {};
        \node (T) at (4, 3) {};

        \node (U) at (5, 0) {};
        \node (V) at (5, 1) {};
        \node (W) at (5, 2) {};
        \node (X) at (5, 3) {};

        \draw (A) -- (F);
        \draw (A) -- (G);
        \draw (A) -- (H);

        \draw (B) -- (E);
        \draw (B) -- (G);
        \draw (B) -- (H);

        \draw (C) -- (H);
        \draw (C) -- (F);
        \draw (C) -- (E);

        \draw (D) -- (G);
        \draw (D) -- (F);
        \draw (D) -- (E);

        \draw (H) -- (K);
        \draw (H) -- (J);
        \draw (H) -- (I);

        \draw (G) -- (L);
        \draw (G) -- (J);
        \draw (G) -- (I);

        \draw (F) -- (L);
        \draw (F) -- (K);
        \draw (F) -- (I);

        \draw (E) -- (L);
        \draw (E) -- (K);
        \draw (E) -- (J);

        \draw (L) -- (O);
        \draw (L) -- (N);
        \draw (L) -- (M);

        \draw (K) -- (P);
        \draw (K) -- (N);
        \draw (K) -- (M);

        \draw (J) -- (P);
        \draw (J) -- (O);
        \draw (J) -- (M);

        \draw (I) -- (P);
        \draw (I) -- (O);
        \draw (I) -- (N);

        \draw (P) -- (S);
        \draw (P) -- (R);
        \draw (P) -- (Q);

        \draw (O) -- (T);
        \draw (O) -- (R);
        \draw (O) -- (Q);

        \draw (N) -- (T);
        \draw (N) -- (S);
        \draw (N) -- (Q);

        \draw (M) -- (T);
        \draw (M) -- (S);
        \draw (M) -- (R);

        \draw (T) -- (W);
        \draw (T) -- (V);
        \draw (T) -- (U);

        \draw (S) -- (X);
        \draw (S) -- (V);
        \draw (S) -- (U);

        \draw (R) -- (X);
        \draw (R) -- (W);
        \draw (R) -- (U);

        \draw (Q) -- (X);
        \draw (Q) -- (W);
        \draw (Q) -- (V);
        
    \end{scope}
    \end{tikzpicture}
    \caption{$B_1$ for $K_4 \times P_6$ consists of the blue vertices.}
    \label{fig:Theorem-2.1.1}
\end{figure}
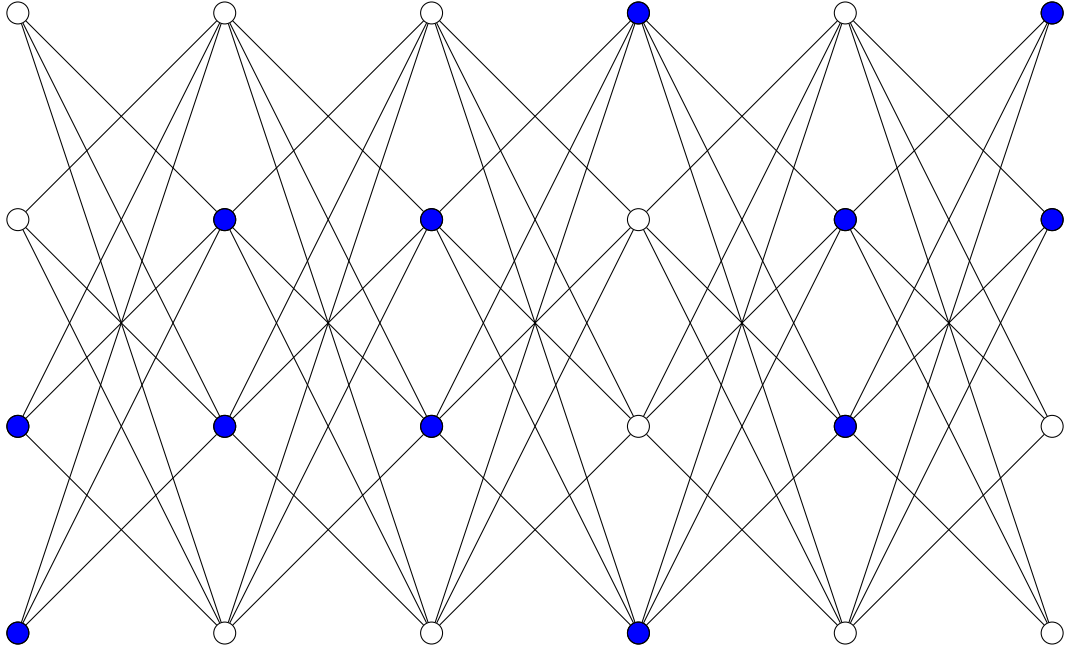

\noindent \textbf{Subcase 1.2: $n$ odd.}

 Let the initially infected set be vertices $\{(j, 1)\}_{j = 3}^n \cup \{(j, 2)\}_{j = 2}^{n - 1} \cup \{(j, i)\}_{i = 2k + 1, j = 2}^{j = n - 1} \cup \{(j, i)\}_{i = 4k, j \notin \{n - 1, n - 2\}} \cup \{(j, i)\}_{i = 4k + 2, i \neq t, j \notin \{2, 3\} }\cup \{(j, t)\}_{j=1}^{n-2}$. Note that this set has size $(n-2)t$ since each of the $t$ columns has $n-2$ initially infected vertices. 

    There are 2 cases to consider: The leak is in column $1$ or $t$ or the leak is in some interior column. 

    \textbf{Subcase 1.2.1: Leak in column $1$ or $t$.}
    
    Without loss of generality, assume the leak is in column $t$. Vertex $(n, 1)$ colors $(1, 2)$. Then any vertex in column $1$ forces $(t, 2)$. Then vertices $(n-2, 3)$ and $(n-1, 3)$ color vertices $(n-1, 4)$ and $(n-2, 4)$, respectively. 

    This coloring pattern repeats with vertices $\{(j, i)\}_{i = 4k - 1, j \in \{n - 1, n - 2\}}$ forcing vertices $\{(j, i)\}_{i = 4k, j \in \{n - 1, n - 2\}}$ and vertices $\{(j, i)\}_{i = 4k + 1, j \in \{2, 3\}}$ forcing vertices $\{(j, i)\}_{i = 4k + 2, j \in \{2, 3\}}$ until all vertices in even-indexed columns except column $t$ are forced. Then vertex $(n - 1, t - 1)$ forces $(n, t)$ and $(n,t-1)$ forces $(n-2,t)$. 
    
    Next, at least one of vertices $(1, t)$ or $(n, t)$ forces $(n, t-1)$ or $(1, t-1)$, respectively. Once one of these is colored, the other may be colored by some other vertex in column $t$. Once columns $t$ and $t-1$ are forced, vertices of the form $(1,2p)$ force vertices $(n,2p-1)$ and $(n,2p)$ force $(1,2p-1)$ in order of decreasing $p$ for all $p$ until only vertices $(1,1)$ and $(2,1)$ remain uncolored.  However, $(2,2)$ and $(1,2)$, respectively, force them. A similar, but mirrored argument holds if the leak is in column $1$. %$\{(j, i)\}_{i = 2k, k > 1, j \in \{1, n\}}^{k = t - 2}$ color vertices $\{(i, j)\}_{i = 2k + 1, k > 0, j \in \{1, n\}}^{k = t - 3}$, and we are left only with vertices $(1, 1)$ and $(1, 2)$ uncolored. We see that these vertices are colored by vertices $(2, 2)$ and $(2, 1)$, respectively. 

    \textbf{Subcase 1.2.2: Leak in an interior column $h$.} 

    Similar to Subcase 1.1.2, even columns force odd ones from the left and odd columns force even from the right until either columns $h-1$ and $h$ are forced ($h$ odd) or $h$ and $h+1$ are forced ($h$ even). Either $h-1$ or $h+1$, depending on the case, can force all vertices in $h-2$ or $h+2$, respectively, allowing vertices in columns $h-3$ or $h+3$ to force $h-4$ or $h+4$. This process repeats until all vertices to the left or right of column $h$ are forced.
    
    Column $h$ can force at least one vertex in the column next to it that is not forced, however once one of these vertices are forced, the other is too, as all but one vertex in column $h$ is in the neighborhood of the remaining unforced vertex. Once this entire column is forced, the remaining uncolored vertices are then forced column by column.
    
    %Since the leak is neither in column $1$ nor $t$, coloring may being in each of these columns. Vertex $(1, n)$ colors vertex $(2, 1)$, and vertex $(t, 1)$ colors vertex $(t - 1, n)$. As in case 1, $(2, n)$ and $(t - 1, 1)$ are then colored.

    %Then one of vertices $(t-2, 1)$, $(t - 2, n)$ may color either $(t - 3, n)$ or $(t - 2, 1)$, respectively. Once on of these is colored, the other may be colored. We see that this pattern repeats until all vertices in each odd column are colored. Then the remaining uncolored vertices in column $1$ may be colored as in subcase 1. 

    %Then we have $2$ uncolored vertices remaining in each column $i = 2k$ where $k > 1$. Then begining with column $3$, we will have $2$ vertices in each odd column which may color $1$ uncolored vertex in the following even column. Then once one of these is colored, the other may be colored, and this continues until the graph is fully colored. 

    \noindent \textbf{Case 2: $t$ odd.}
    We consider two subcases: $n$ even and $n$ odd.

    \noindent \textbf{Subcase 2.1: $n$ even.}
    
    Let the initially infected set be vertices $\{(i,1)\}_{i =3}^n \cup \{(j, t-1)\}_{j=1}^{n-2}  \cup \{(f, 2)\}_{f \notin \{1,n\}} \cup \{ (m, 2k) \}_{k \notin \{1, \frac{t-1}{2}\}, m \notin \{\frac{n}{2}, \frac{n}{2}+1\}} \cup \{(q,2p+1 )\}_{p \neq 0, q \notin \{1,n\}} \cup \{(i,t)\}_{i =1}^n$. Note that this set has size $(n-2)t + 2$ since each column except for column $t$ has $n-2$ initially infected vertices and column $t$ has $n$ initially infected vertices. 

Note that this set is identical to that in Subcase 1.1, with an additional column that is fully infected (see Fig.~\ref{fig:case2.1}) . Since this last column is in $B_1$, it does not impact how the column next to it infects other vertices. Thus, this set is sufficient to force the entire graph.
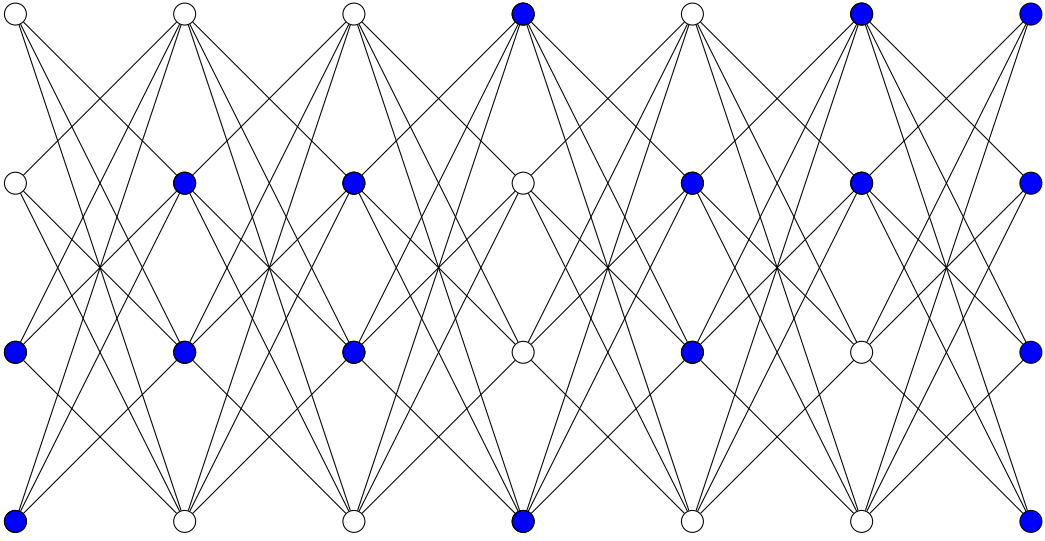
\begin{figure}
    \centering
    \begin{tikzpicture} [scale=2.25]
    \begin{scope} [every node/.style={scale=.75,circle,draw}]
    \begin{scope} [every node/.style={scale=.75,circle,draw, fill=blue}]
    
        \node (A) at (0, 0) {};
        \node (B) at (0, 1) {};
        \node (F) at (1, 1) {};
        \node (G) at (1, 2) {};
        \node (J) at (2, 1) {};
        \node (K) at (2, 2) {};
        \node (P) at (3, 3) {};
        \node (M) at (3, 0) {};
        \node (R) at (4, 1) {};
        \node (S) at (4, 2) {};
        \node (W) at (5, 2) {};
        \node (X) at (5, 3) {};
        \node (Y) at (6, 0) {};
        \node (Z) at (6, 1) {};
        \node (AA) at (6, 2) {};
        \node (AB) at (6, 3) {};
        
    \end{scope}
    
        \node (A) at (0, 0) {};
        \node (B) at (0, 1) {};
        \node (C) at (0, 2) {};
        \node (D) at (0, 3) {};

        \node (E) at (1, 0) {};
        \node (F) at (1, 1) {};
        \node (G) at (1, 2) {};
        \node (H) at (1, 3) {};

        \node (I) at (2, 0) {};
        \node (J) at (2, 1) {};
        \node (K) at (2, 2) {};
        \node (L) at (2, 3) {};

        \node (M) at (3, 0) {};
        \node (N) at (3, 1) {};
        \node (O) at (3, 2) {};
        \node (P) at (3, 3) {};

        \node (Q) at (4, 0) {};
        \node (R) at (4, 1) {};
        \node (S) at (4, 2) {};
        \node (T) at (4, 3) {};

        \node (U) at (5, 0) {};
        \node (V) at (5, 1) {};
        \node (W) at (5, 2) {};
        \node (X) at (5, 3) {};

        \draw (A) -- (F);
        \draw (A) -- (G);
        \draw (A) -- (H);

        \draw (B) -- (E);
        \draw (B) -- (G);
        \draw (B) -- (H);

        \draw (C) -- (H);
        \draw (C) -- (F);
        \draw (C) -- (E);

        \draw (D) -- (G);
        \draw (D) -- (F);
        \draw (D) -- (E);

        \draw (H) -- (K);
        \draw (H) -- (J);
        \draw (H) -- (I);

        \draw (G) -- (L);
        \draw (G) -- (J);
        \draw (G) -- (I);

        \draw (F) -- (L);
        \draw (F) -- (K);
        \draw (F) -- (I);

        \draw (E) -- (L);
        \draw (E) -- (K);
        \draw (E) -- (J);

        \draw (L) -- (O);
        \draw (L) -- (N);
        \draw (L) -- (M);

        \draw (K) -- (P);
        \draw (K) -- (N);
        \draw (K) -- (M);

        \draw (J) -- (P);
        \draw (J) -- (O);
        \draw (J) -- (M);

        \draw (I) -- (P);
        \draw (I) -- (O);
        \draw (I) -- (N);

        \draw (P) -- (S);
        \draw (P) -- (R);
        \draw (P) -- (Q);

        \draw (O) -- (T);
        \draw (O) -- (R);
        \draw (O) -- (Q);

        \draw (N) -- (T);
        \draw (N) -- (S);
        \draw (N) -- (Q);

        \draw (M) -- (T);
        \draw (M) -- (S);
        \draw (M) -- (R);

        \draw (T) -- (W);
        \draw (T) -- (V);
        \draw (T) -- (U);

        \draw (S) -- (X);
        \draw (S) -- (V);
        \draw (S) -- (U);

        \draw (R) -- (X);
        \draw (R) -- (W);
        \draw (R) -- (U);

        \draw (Q) -- (X);
        \draw (Q) -- (W);
        \draw (Q) -- (V);

        \draw (U) -- (Z);
        \draw (U) -- (AA);
        \draw (U) -- (AB);
                
        \draw (V) -- (Y);
        \draw (V) -- (AA);
        \draw (V) -- (AB);

        \draw (W) -- (Y);
        \draw (W) -- (Z);
        \draw (W) -- (AB);
                
        \draw (X) -- (Y);
        \draw (X) -- (AA);
        \draw (X) -- (Z);
        
    \end{scope}
    \end{tikzpicture}
    \caption{$B_1$ for $K_4 \times P_7$ consists of the blue vertices. This contains $B_1$ for $K_4 \times P_6$ in the first six columns.}
    \label{fig:case2.1}
\end{figure}

\noindent \textbf{Case 2.2: $n$ odd.} 

Let the initially infected set be vertices $\{(1, j)\}_{j = 3}^n \cup \{(2, j)\}_{j = 2}^{n - 1} \cup \{(i, j)\}_{i = 2k + 1 \neq t, j = 2}^{j = n - 1} \cup \{(i, j\}_{i = 4k, j \notin \{n - 1, n - 2\}} \cup \{(i, j)\}_{i = 4k + 2, j \notin \{2, 3\} }\cup \{(t, j)\}_{j=1}^{n}$. Note that this set has size $(n-2)t + 2$ since each of the $t$ columns has $(n-2)$ colored vertices and column $t$ has an additional $2$ colored vertices. 

This set is identical to that in Subcase 1.2 with an additional column $t$ which is fully infected. Since this last column is in $B_1$, it does not impact how the column next to it infects other vertices. Thus, this set is sufficient to $1$-force the graph.

\end{proof}

In general, we conjecture that $K_n \times P_t$ is not 2-resilient. For an intuition as to why, consider Fig.~\ref{fig:not2resilient}. As one can see, (using the same vertex labeling convention from the theorem) only $(2,1)$ and $(5,4)$ are infected after one application of the color change rule. No more vertices can be infected after this. With this configuration, we need at least one end column to force every vertex in the column next to it, implying all vertices in an end column are initially infected. There may, however, be other $l$-leaky forcing sets of size $(n-2)t$ or $(n-2)t+2$, depending on the parity of $t$, that contain all vertices in both end columns in an initial leaky forcing set and fewer vertices in the interior columns. However, if there are interior columns with fewer initially colored vertices, they might not be able to be forced since the vertices in columns on either side of it now have more neighbors that are not colored, so fewer of them might be able to force. The authors of this work have not been able to find such a set. Thus, we conjecture that a $2$-leaky forcing set of $K_n \times P_t$ is larger than $(n-2)t$ or $(n-2)t+2$, implying $K_n \times P_t$ is not 2-resilient.
\begin{figure}
    \centering
    \begin{tikzpicture} [scale=2.75]
    \begin{scope} [every node/.style={scale=.75,circle,draw}]
    \begin{scope} [every node/.style={scale=.75,circle,draw, fill=blue}]
    
        \node (A) at (0, 0) {};
        %\node (B) at (0, 1) {};
        \node (F) at (1, 1) {};
        \node (G) at (1, 2) {};
        \node (J) at (2, 1) {};
        \node (K) at (2, 2) {};
        \node (P) at (3, 3) {};
        \node (M) at (3, 0) {};
        \node (R) at (4, 1) {};
        \node (S) at (4, 2) {};
        %\node (W) at (5, 2) {};
        \node (X) at (5, 3) {};
        
    \end{scope}
      \begin{scope} [every node/.style={scale=.75,circle,draw, fill=red}]
    
        \node (B) at (0, 1) {};
        \node (W) at (5, 2) {};
        
    \end{scope}
    
        \node (A) at (0, 0) {};
        \node (B) at (0, 1) {};
        \node (C) at (0, 2) {};
        \node (D) at (0, 3) {};

        \node (E) at (1, 0) {};
        \node (F) at (1, 1) {};
        \node (G) at (1, 2) {};
        \node (H) at (1, 3) {};

        \node (I) at (2, 0) {};
        \node (J) at (2, 1) {};
        \node (K) at (2, 2) {};
        \node (L) at (2, 3) {};

        \node (M) at (3, 0) {};
        \node (N) at (3, 1) {};
        \node (O) at (3, 2) {};
        \node (P) at (3, 3) {};

        \node (Q) at (4, 0) {};
        \node (R) at (4, 1) {};
        \node (S) at (4, 2) {};
        \node (T) at (4, 3) {};

        \node (U) at (5, 0) {};
        \node (V) at (5, 1) {};
        \node (W) at (5, 2) {};
        \node (X) at (5, 3) {};

        \draw (A) -- (F);
        \draw (A) -- (G);
        \draw (A) -- (H);

        \draw (B) -- (E);
        \draw (B) -- (G);
        \draw (B) -- (H);

        \draw (C) -- (H);
        \draw (C) -- (F);
        \draw (C) -- (E);

        \draw (D) -- (G);
        \draw (D) -- (F);
        \draw (D) -- (E);

        \draw (H) -- (K);
        \draw (H) -- (J);
        \draw (H) -- (I);

        \draw (G) -- (L);
        \draw (G) -- (J);
        \draw (G) -- (I);

        \draw (F) -- (L);
        \draw (F) -- (K);
        \draw (F) -- (I);

        \draw (E) -- (L);
        \draw (E) -- (K);
        \draw (E) -- (J);

        \draw (L) -- (O);
        \draw (L) -- (N);
        \draw (L) -- (M);

        \draw (K) -- (P);
        \draw (K) -- (N);
        \draw (K) -- (M);

        \draw (J) -- (P);
        \draw (J) -- (O);
        \draw (J) -- (M);

        \draw (I) -- (P);
        \draw (I) -- (O);
        \draw (I) -- (N);

        \draw (P) -- (S);
        \draw (P) -- (R);
        \draw (P) -- (Q);

        \draw (O) -- (T);
        \draw (O) -- (R);
        \draw (O) -- (Q);

        \draw (N) -- (T);
        \draw (N) -- (S);
        \draw (N) -- (Q);

        \draw (M) -- (T);
        \draw (M) -- (S);
        \draw (M) -- (R);

        \draw (T) -- (W);
        \draw (T) -- (V);
        \draw (T) -- (U);

        \draw (S) -- (X);
        \draw (S) -- (V);
        \draw (S) -- (U);

        \draw (R) -- (X);
        \draw (R) -- (W);
        \draw (R) -- (U);

        \draw (Q) -- (X);
        \draw (Q) -- (W);
        \draw (Q) -- (V);
        
    \end{scope}
    \end{tikzpicture}
    \caption{An intuition as to why $K_n \times P_t$ is not 2-resilient. $B_1$ consists of the blue and red vertices, however the red vertices are leaky.}
    \label{fig:not2resilient}
\end{figure}
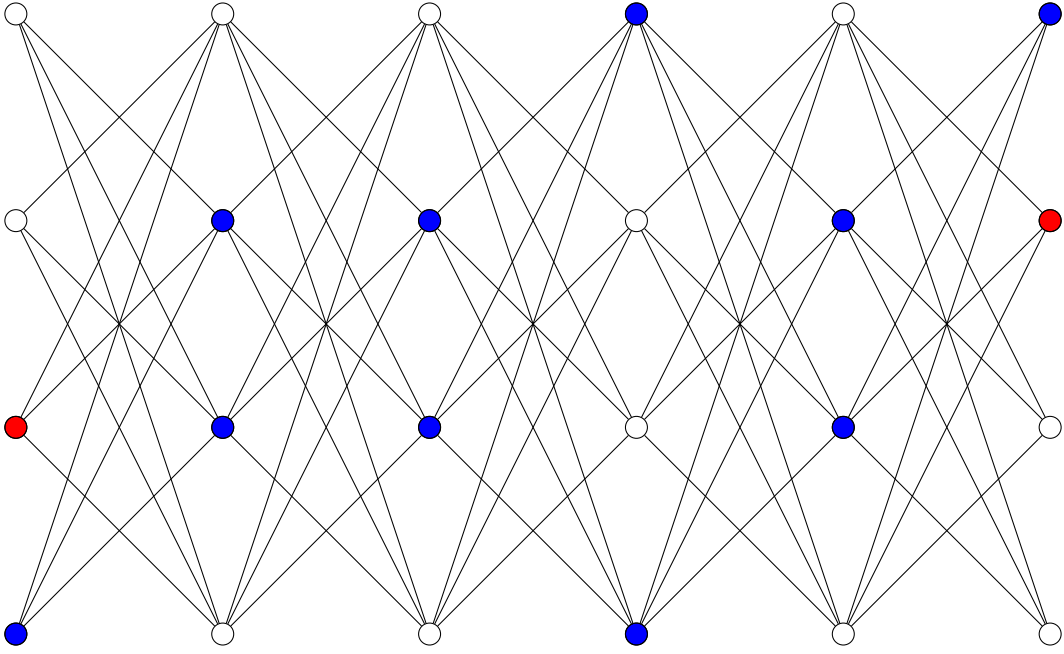

\begin{conjecture}
    $K_n \times P_t$ is not 2-resilient.
\end{conjecture}

The graph $K_n \times C_t$, where $C_t$ denotes a cycle on $t$ vertices, contains $K_n \times P_t$ as a subgraph, but with the addition of $n(n-1)$ additional edges that connect vertices of the form $(1,i)$ and $(t,j)$ for $ij \in E(K_n)$. Next, we show $K_n \times C_t$ is $1$-resilient.

\begin{thm}
      $K_n \times C_t$ is $1$-resilient, thus 
    \[Z_1(K_n \times C_t) = Z_0(K_n \times C_t) = \begin{cases} 
           (n-2)t + 2 & t \; \mathrm{odd} \\
           (n-2)t + 4 & t\; \mathrm{even} \\
            \end{cases}.
    \]
\end{thm}
    
\begin{proof}
Label the vertices of $K_n \times C_t$ as in the previous theorem. Benson, Ferrero, Flagg, Furst, Hogben, Vasilevska, and Wissman showed that $Z_0(K_n \times C_t) = (n-2)t + 4$ for even $t$, and $Z_0(K_n \times C_t) = (n-2)t + 2$ for odd $t$ \cite{benson2018zero} so we need only show that there is a construction of a set of size $(n-2)t + 2$ that $1$-forces for odd $t$, and a set of size  $(n-2)t + 4$ for even $t$.

There are $4$ cases: $t$ and $n$ are both even or both odd, or $t$ is even and $n$ is odd or $n$ is even and $t$ is odd.

\textbf{Case 1: $t$ is even and $n$ is even}

Let the initially infected set of vertices be $\{(i,1)\}_{i =1}^n \cup \{(j, t)\}_{j=1}^{n}  \cup \{(f, 2)\}_{f \notin \{1,n\}} \cup \{ (m, 2k) \}_{k \notin \{1, \frac{t}{2}\}, m \notin \{\frac{n}{2}, \frac{n}{2}+1\}} \cup \{(q,2p+1)\}_{p \neq 0, q \notin \{1,n\}}$. Note that this set has size $(n-2)t + 4$ since each of the $t$ columns has $n-1$ colored vertices, and columns $1$ and $t$ each have $2$ additional colored vertices. Then since columns $1$ and $t$ are fully colored, the coloring process for the interior columns is identical to that in Theorem~\ref{thm:pathcomplete} Subcase 1.1, so the graph may be fully colored. 

\textbf{Case 2: $t$ is even and $n$ is odd}

Let the initially infected set of vertices be $\{(j, 1)\}_{j = 1}^n \cup \{(j, 2)\}_{j = 2}^{n - 1} \cup \{(j, i)\}_{i = 2k + 1, j = 2}^{j = n - 1} \cup \{(j, i\}_{i = 4k, j \notin \{n - 1, n - 2\}} \cup \{(j, i)\}_{i = 4k + 2, i \neq t, j \notin \{2, 3\} }\cup \{(j, t)\}_{j=1}^{n}$. Note that this set has size $(n-2)t + 4$ since each of the $t$ columns has $n-1$ colored vertices, and columns $1$ and $t$ each have $2$ additional colored vertices. Then since columns $1$ and $t$ are fully colored, the coloring process for the interior columns is identical to that in Theorem~\ref{thm:pathcomplete} Subcase 1.2, so the graph may be fully colored. 

\textbf{Case 3: $t$ is odd and $n$ is even}

Let the initially infected set of vertices be $\{(i,1)\}_{i =3}^n \cup \{(j, t-1)\}_{j=1}^{n-2}  \cup \{(f, 2)\}_{f \notin \{1,n\}} \cup \{ (m, 2k) \}_{k \notin \{1, \frac{t-1}{2}\}, m \notin \{\frac{n}{2}, \frac{n}{2}+1\}} \cup \{(q,2p+1 )\}_{p \neq 0, q \notin \{1,n\}} \cup \{(i,t)\}_{i =1}^n$. Then this set has size $(n-2)t + 2$ since each of the $t$ columns has $n-2$ colored vertices, and column $t$ has $2$ additional colored vertices. Then since every vertex in column $t$ is colored, vertices in columns $1$ and $t-1$ have no uncolored neighbors in column $t$, and the coloring process is identical to that in Theorem~\ref{thm:pathcomplete} Subcase 2.1. Thus, this set can force the entire graph.

\textbf{Case 4: $t$ is odd and $n$ is odd}
Let the initially infected set be vertices $\{(j, 1)\}_{j = 3}^n \cup \{(j, 2)\}_{j = 2}^{n - 1} \cup \{(j, i)\}_{i = 2k + 1 \neq t, j = 2}^{j = n - 1} \cup \{(j, i\}_{i = 4k, j \notin \{n - 1, n - 2\}} \cup \{(j, i)\}_{i = 4k + 2, j \notin \{2, 3\} }\cup \{(j, t)\}_{j=1}^{n}$. Note that this set has size $(n-2)t + 2$ since each of the $t$ columns has $(n-2)$ colored vertices and column $t$ has an additional $2$ colored vertices. Then since column $t$ is fully colored, vertices in columns $1 and t-1$ have no uncolored neighbors in column $t$, and we have the same coloring pattern as in Theorem~\ref{thm:pathcomplete} Subcase 2.2. 
\end{proof}

Unlike $K_n \times P_t$, $K_n \times C_t$ could potentially be $2$-resilient for even $t$. This is because in the constructed $1$-forcing set above, two columns of vertices in the graph that are next to each other are both completely colored. Thus, if there is a single leak in both of these columns, there still exist vertices in these columns that can force others. This argument is less convincing for odd $t$, so it would be interesting to consider if the parity of $t$ impacts the resilience of the product. Thus, we ask

\begin{question}
Is $K_n \times C_t$ $2$-resilient for either even $t$, odd $t$, both, or neither?
\end{question}

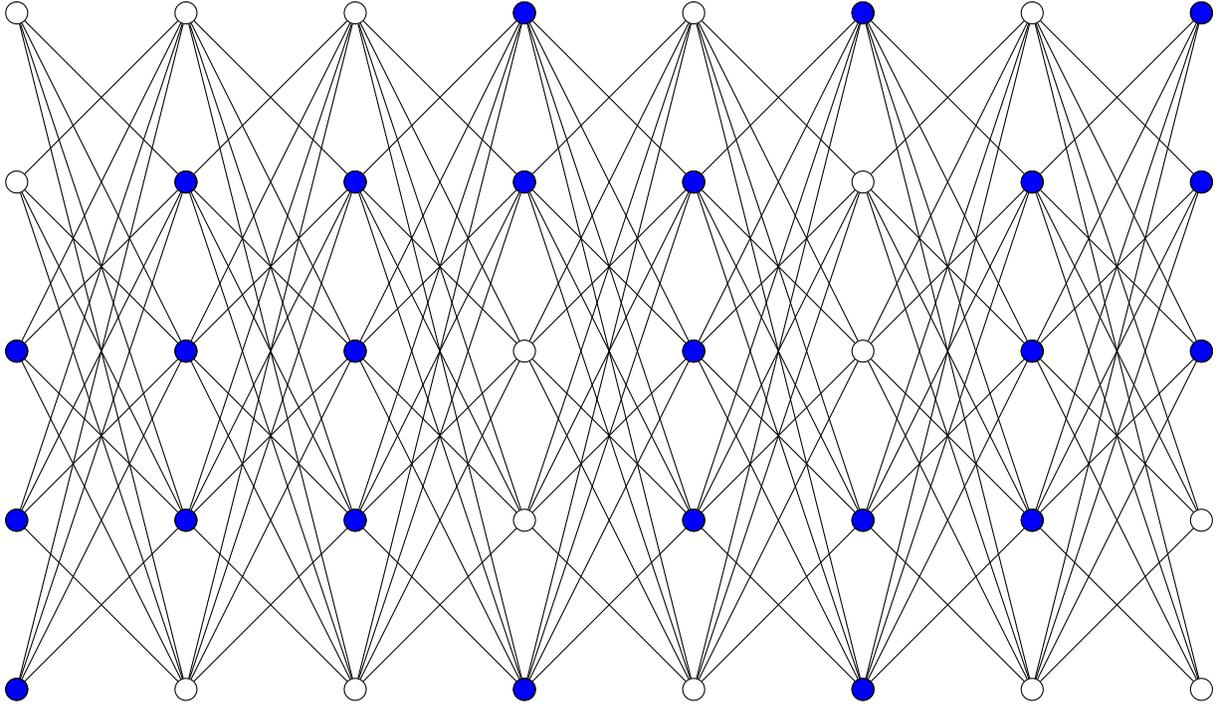
\begin{figure}
    \centering
    \begin{tikzpicture} [scale=2.25]
    \begin{scope} [every node/.style={scale=.75,circle,draw}]
    \begin{scope} [every node/.style={scale=.75,circle,draw, fill=blue}]

        \node (A) at (0, 0) {};
        \node (B) at (0, 1) {};
        \node (C) at (0, 2) {};
        \node (F) at (1, 1) {};
        \node (G) at (1, 2) {};
        \node (H) at (1, 3) {};
        \node (J) at (2, 1) {};
        \node (K) at (2, 2) {};
        \node (L) at (2, 3) {};
        \node (X) at (3, 4) {};
        \node (P) at (3, 3) {};
        \node (M) at (3, 0) {};
        \node (R) at (4, 1) {};
        \node (S) at (4, 2) {};
        \node (T) at (4, 3) {};
        \node (5) at (5, 4) {};
        \node (1) at (5, 0) {};
        \node (2) at (5, 1) {};
        \node (7) at (6, 1) {};
        \node (8) at (6, 2) {};
        \node (9) at (6, 3) {};
        \node (13) at (7, 2) {};
        \node (14) at (7, 3) {};
        \node (15) at (7, 4) {};
      
    \end{scope}
    
        \node (A) at (0, 0) {};
        \node (B) at (0, 1) {};
        \node (C) at (0, 2) {};
        \node (D) at (0, 3) {};

        \node (E) at (1, 0) {};
        \node (F) at (1, 1) {};
        \node (G) at (1, 2) {};
        \node (H) at (1, 3) {};

        \node (I) at (2, 0) {};
        \node (J) at (2, 1) {};
        \node (K) at (2, 2) {};
        \node (L) at (2, 3) {};

        \node (M) at (3, 0) {};
        \node (N) at (3, 1) {};
        \node (O) at (3, 2) {};
        \node (P) at (3, 3) {};

        \node (Q) at (4, 0) {};
        \node (R) at (4, 1) {};
        \node (S) at (4, 2) {};
        \node (T) at (4, 3) {};

        \node (U) at (0, 4) {};
        \node (V) at (1, 4) {};
        \node (W) at (2, 4) {};
        \node (X) at (3, 4) {};
        \node (Y) at (4, 4) {};

        \node (1) at (5, 0) {};
        \node (2) at (5, 1) {};
        \node (3) at (5, 2) {};
        \node (4) at (5, 3) {};
        \node (5) at (5, 4) {};

        \node (6) at (6, 0) {};
        \node (7) at (6, 1) {};
        \node (8) at (6, 2) {};
        \node (9) at (6, 3) {};
        \node (10) at (6, 4) {};

        \node (11) at (7, 0) {};
        \node (12) at (7, 1) {};
        \node (13) at (7, 2) {};
        \node (14) at (7, 3) {};
        \node (15) at (7, 4) {};

        \draw (A) -- (F);
        \draw (A) -- (G);
        \draw (A) -- (H);

        \draw (B) -- (E);
        \draw (B) -- (G);
        \draw (B) -- (H);

        \draw (C) -- (H);
        \draw (C) -- (F);
        \draw (C) -- (E);

        \draw (D) -- (G);
        \draw (D) -- (F);
        \draw (D) -- (E);

        \draw (H) -- (K);
        \draw (H) -- (J);
        \draw (H) -- (I);

        \draw (G) -- (L);
        \draw (G) -- (J);
        \draw (G) -- (I);

        \draw (F) -- (L);
        \draw (F) -- (K);
        \draw (F) -- (I);

        \draw (E) -- (L);
        \draw (E) -- (K);
        \draw (E) -- (J);

        \draw (L) -- (O);
        \draw (L) -- (N);
        \draw (L) -- (M);

        \draw (K) -- (P);
        \draw (K) -- (N);
        \draw (K) -- (M);

        \draw (J) -- (P);
        \draw (J) -- (O);
        \draw (J) -- (M);

        \draw (I) -- (P);
        \draw (I) -- (O);
        \draw (I) -- (N);

        \draw (P) -- (S);
        \draw (P) -- (R);
        \draw (P) -- (Q);

        \draw (O) -- (T);
        \draw (O) -- (R);
        \draw (O) -- (Q);

        \draw (N) -- (T);
        \draw (N) -- (S);
        \draw (N) -- (Q);

        \draw (M) -- (T);
        \draw (M) -- (S);
        \draw (M) -- (R);

        \draw (U) -- (H);
        \draw (U) -- (G);
        \draw (U) -- (F);
        \draw (U) -- (E);

        \draw (V) -- (D);
        \draw (V) -- (C);
        \draw (V) -- (B);
        \draw (V) -- (A);
        \draw (V) -- (L);
        \draw (V) -- (K);
        \draw (V) -- (J);
        \draw (V) -- (I);

        \draw (W) -- (H);
        \draw (W) -- (G);
        \draw (W) -- (F);
        \draw (W) -- (E);
        \draw (W) -- (P);
        \draw (W) -- (O);
        \draw (W) -- (N);
        \draw (W) -- (M);

        \draw (X) -- (L);
        \draw (X) -- (K);
        \draw (X) -- (J);
        \draw (X) -- (I);
        \draw (X) -- (T);
        \draw (X) -- (S);
        \draw (X) -- (R);
        \draw (X) -- (Q);

        \draw (Y) -- (P);
        \draw (Y) -- (O);
        \draw (Y) -- (N);
        \draw (Y) -- (M);
        
        \draw (Y) -- (4);
        \draw (Y) -- (3);
        \draw (Y) -- (2);
        \draw (Y) -- (1);

        \draw (T) -- (5);
        \draw (T) -- (3);
        \draw (T) -- (2);
        \draw (T) -- (1);

        \draw (S) -- (5);
        \draw (S) -- (4);
        \draw (S) -- (2);
        \draw (S) -- (1);

        \draw (R) -- (5);
        \draw (R) -- (4);
        \draw (R) -- (3);
        \draw (R) -- (1);

        \draw (Q) -- (5);
        \draw (Q) -- (4);
        \draw (Q) -- (3);
        \draw (Q) -- (2);

        \draw (5) -- (9);
        \draw (5) -- (8);
        \draw (5) -- (7);
        \draw (5) -- (6);

        \draw (4) -- (10);
        \draw (4) -- (8);
        \draw (4) -- (7);
        \draw (4) -- (6);

        \draw (3) -- (10);
        \draw (3) -- (9);
        \draw (3) -- (7);
        \draw (3) -- (6);

        \draw (2) -- (10);
        \draw (2) -- (9);
        \draw (2) -- (8);
        \draw (2) -- (6);

        \draw (1) -- (10);
        \draw (1) -- (9);
        \draw (1) -- (8);
        \draw (1) -- (7);

        \draw (10) -- (14);
        \draw (10) -- (13);
        \draw (10) -- (12);
        \draw (10) -- (11);

        \draw (9) -- (15);
        \draw (9) -- (12);
        \draw (9) -- (13);
        \draw (9) -- (11);

        \draw (8) -- (15);
        \draw (8) -- (14);
        \draw (8) -- (12);
        \draw (8) -- (11);

        \draw (7) -- (15);
        \draw (7) -- (14);
        \draw (7) -- (13);
        \draw (7) -- (11);

        \draw (6) -- (15);
        \draw (6) -- (14);
        \draw (6) -- (13);
        \draw (6) -- (12);

    \end{scope}
    \end{tikzpicture}
    \caption{$B_1$ for $K_5 \times P_8$ consists of the blue vertices.}
    \label{fig:Theorem-2.1.2}
\end{figure}

Next, we turn our attention to products of $K_n$ with itself.

\begin{thm}
    $Z_{(1)}(K_n \times  K_n) = n(n-2) +2(n-2) = n^2-4$ for all $n \geq 3$.
\end{thm}

\begin{proof}
In \cite{huang2010minimum}, Huang, Chang, and Yeh show that $Z_0(K_n \times  K_n) = n^2-4$, thus we need only find a construction of a $1-$leaky forcing set of this cardinality to prove the result.

Label the vertices of $K_n \times K_n$ as $(i,j)$ where $i,j \in [n].$ For $n \geq 4$, add all vertices to $B_1$ except $(1,n-1)$, $(1,n)$, $(n,1)$ and $(n, 2)$. The vertex $(1,2)$ can force $(n,1)$, $(1,1)$ can force $(n,2)$, $(n,n)$ can force $(1,n-1)$, and $(n,n-1)$ can force $(1,n)$. Since there is at most one leak, at least three of these vertices not in $B_1$ are forced in the first step. Once three of these vertices have been forced, the last is forced in the next step, as all but one of its $(n-1)^2$ neighbors can then force it.

For $n=3$, consider the set colored in Fig.~\ref{fig:k3timesk3}. With this initial set, either $(2,1)$, $(2,3)$, or both are colored after the first application of the color change rule (colored by either $(3,3)$ or $( 1,1 )$ respectively). If both are colored, $(1,3)$ and $(3,1)$ are both colored on the next application of the color change rule since they each have two colored neighbors and no neighbors in common. If only $(2,1)$ is colored after the first application of the color change rule, then $(1,1)$ is leaky. $(2,1)$ can then infect $(1,3)$ on the second application of the color change rule, then $(1,3)$ can infect $(3,1)$, which in turn can infect $(2,3)$. Similarly, if only $(2,1)$ is infected in the first application of the color change rule, then $(3,3)$ has the leak. $(1,1)$ can then infect $(2,3)$, which in turn infect $(3,1)$. Then $(3,1)$ can infect $(1,3)$.
\end{proof}
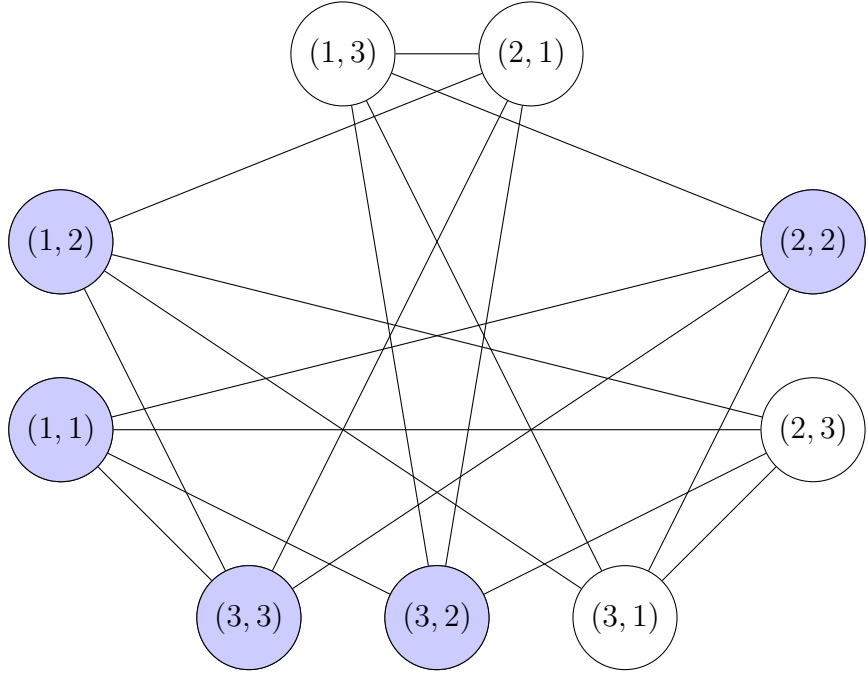
\begin{figure}
    \centering
    \begin{tikzpicture} [scale=2.5]
    \begin{scope} [every node/.style={scale=1,circle,draw}]
    
        \node (A) at (-0.5, 0) {$(1,1)$};
        \node (B) at (-0.5, 1) {$(1,2)$};
        \node (C) at (1, 2) {$(1,3)$};
        \node (D) at (2, 2) {$(2,1)$};
        \node (F) at (3.5, 1) {$(2,2)$};
        \node (G) at (3.5, 0) {$(2,3)$};
        \node (H) at (2.5, -1) {$(3,1)$};
        \node (J) at (1.5, -1) {$(3,2)$};
        \node (K) at (0.5, -1) {$(3,3)$};
         
    \end{scope}

    \begin{scope} [every node/.style={scale=1,circle,draw, fill=blue!20}]
    
        \node (A) at (-0.5, 0) {$(1,1)$};
        \node (B) at (-0.5, 1) {$(1,2)$};
        \node (F) at (3.5, 1) {$(2,2)$};
        \node (J) at (1.5, -1) {$(3,2)$};
        \node (K) at (0.5, -1) {$(3,3)$};
         
    \end{scope}

        \draw (A) -- (F);
        \draw (A) -- (G);
        \draw (A) -- (J);
         \draw (A) -- (K);

         \draw (B) -- (D);
        \draw (B) -- (G);
        \draw (B) -- (H);
         \draw (B) -- (K);

        \draw (C) -- (F);
        \draw (C) -- (D);
        \draw (C) -- (H);
         \draw (C) -- (J);

        \draw (D) -- (J);
        \draw (D) -- (K);
        \draw (F) -- (H);
         \draw (F) -- (K);
          \draw (G) -- (H);
         \draw (G) -- (J);

    \end{tikzpicture}
    \caption{$K_3 \times K_3$. The blue shaded vertices form a (not unique) $1$-leaky forcing set.}
    \label{fig:k3timesk3}
\end{figure}

Clearly, $Z_0(K_n) = n-1$. Hence, the previous result shows that in general, $Z_{(\ell)}(G \times H)$ may not necessarily be less than $Z_{(\ell)}(G)Z_{(\ell)}(H)$. Thus, we ask
\begin{question}
    Does there exist any two connected graphs $G$ and $H$, aside from $G = K_2$ and $H$ bipartite, such that $Z_{(\ell)}(G \times H) \leq Z_{(\ell)}(G)Z_{(\ell)}(H)$?
\end{question}
When $G = K_2$ and when $H$ is bipartite, $G \times H$ results in two disjoint copies of $H$, which clearly satisfies the inequality. If the question holds true for other pairs of graphs, a natural research topic would be to characterize all graphs that satisfy this property.

\section{Conclusion}\label{sec:conclusion}
In this work, we showed that $K_n \times P_t$ and $K_t \times C_t$ are both $1$-resilient and conjecture that they are not $2$-resilient. Interestingly, every Cartesian product $K_n \times H$ for $n \geq 3$ that we have considered has been $1$-resilient. Thus, we ask 

\begin{question}
    Are all Cartesian products $K_n \times H$ for $n \geq 3$ $1$-resilient?
\end{question}
In order to prove this is not true, one need show that there exists a graph $H$ such that $Z_{(1)}(K_n \times H) > Z_{(0)}(K_n \times H)$.

%After discussing Cartesian products of $K_n$, we then turned our attention to direct products of $K_2$ and showed that $Z_{(d-1)}(Q_d) \leq 3\times 2^{d-2}$. Note that
%\begin{equation*}
 %   2^{d-2} = Z_{(d-2)}(Q_d) \leq Z_{(d-1)}(Q_d) \leq 3\times 2^{d-2}.
%\end{equation*}
%However,  $Z_{(d-1)} (Q_d) = 2^{d-1} + 2$ for small $d$. We conjecture that the true value of $Z_{(d-1)}(Q_d)$ is closer to $2^{d-2}$ than $3\times 2^{d-2}$, and we leave this as an open question:
In \cite{dillman2019leaky}, Dillman and Kenter show that $Z_{(2)}(Q_3) = 6$ and $Z_{(3)}(Q_4) = 10$, the latter shown by exhaustive computer search. Both of these values are equal to $2^{d-2}+2$. Thus, we ask
\begin{question}
    Does $Z_{(d-1)}(Q_d) = 2^{d-2} + 2$ for all $d \geq 3$?
\end{question}

In fact, $Q_5$ appears to have a $4$-leaky forcing set of cardinality eighteen. Consider the initially infected set found in Fig.~\ref{fig:cube5}. This set appears to be $4$-leaky forcing, however a formal proof or exhaustive computer search is still needed to verify this. If it is $4$-leaky forcing, it is interesting to note that this blue set contains a $3$-leaky forcing set of $Q_4$, which in turn contains a $2$-leaky forcing set of $Q_3$. This leads to another question
\begin{question}
    Does there exist an infinite family of graphs, $F_n$, such that $$B_0 (F_k) \subset B_1(F_{k+1}) \subset B_2(F_{k+2}) \ldots \mathrm{?}$$
\end{question}

    \begin{figure}
\begin{tikzpicture}[scale=1.5]
	\begin{scope}[every node/.style={scale=0.75,circle,fill=blue,draw}]
    \node (A) at (0,0) {};
    \node (B) at (0,2) {};
	\node (C) at (2,0) {}; 
	\node (D) at (2,2) {};
	\node (E) at (1,2.6) {}; 
	\node (F) at (1,0.6) {};
%	\node (G) at (3,.6) {}; 
%	\node (H) at (3,2.6) {};

    \node (I) at (5,0) {};
    \node (J) at (5,2) {};
	%\node (K) at (7,0) {}; 
	\node (L) at (7,2) {};
	%\node (M) at (6,2.6) {}; 
	\node (N) at (6,0.6) {};
	\node (O) at (8,.6) {}; 
	%\node (P) at (8,2.6) {};

    % \node (Q) at (0,-4.6) {};
    %\node (R) at (0,-2.6) {};
	\node (S) at (2,-4.6) {}; 
	\node (T) at (2,-2.6) {};
	%\node (U) at (1,-2) {}; 
	\node (V) at (1,-4) {};
	\node (W) at (3,-2) {}; 
	%\node (X) at (3,-4) {};

    %\node (Y) at (5,-4.6) {};
    \node (Z) at (5,-2.6) {};
	%\node (AA) at (7,-4.6) {}; 
	\node (AB) at (7,-2.6) {};
	%\node (AC) at (6,-2) {}; 
	%\node (AD) at (6,-4) {};
	\node (AE) at (8,-2) {}; 
	%\node (AF) at (8,-4) {};

\end{scope}

	\begin{scope}[every node/.style={scale=0.75,circle,fill=none,draw}]
   % \node (A) at (0,0) {};
    %\node (B) at (0,2) {};
	%\node (C) at (2,0) {}; 
	%\node (D) at (2,2) {};
	\node (E) at (1,2.6) {}; 
	\node (F) at (1,0.6) {};
	\node (G) at (3,.6) {}; 
	\node (H) at (3,2.6) {};
        \node (K) at (7,0) {};
        \node (M) at (6,2.6) {}; 
        \node (P) at (8,2.6) {};

             \node (Q) at (0,-4.6) {};
    \node (R) at (0,-2.6) {};
	\node (S) at (2,-4.6) {}; 
	\node (T) at (2,-2.6) {};
	\node (U) at (1,-2) {}; 
	\node (V) at (1,-4) {};
	\node (W) at (3,-2) {}; 
	\node (X) at (3,-4) {};

    \node (Y) at (5,-4.6) {};
    \node (Z) at (5,-2.6) {};
	\node (AA) at (7,-4.6) {}; 
	\node (AB) at (7,-2.6) {};
	\node (AC) at (6,-2) {}; 
	\node (AD) at (6,-4) {};
	\node (AE) at (8,-2) {}; 
	\node (AF) at (8,-4) {};

\end{scope}
		
        \draw    (A) to[out=-20,in=200] (I);
        \draw    (B) to[out=-20,in=200] (J);
        \draw    (C) to[out=-20,in=200] (K);
        \draw    (D) to[out=-20,in=200] (L);
        \draw    (E) to[out=20,in=160] (M);
        \draw    (F) to[out=20,in=160] (N);
        \draw    (G) to[out=20,in=160] (O);
        \draw    (H) to[out=20,in=160] (P);

        \draw    (Q) to[out=-20,in=200] (Y);
        \draw    (R) to[out=-20,in=200] (Z);
        \draw    (S) to[out=-20,in=200] (AA);
        \draw    (T) to[out=-20,in=200] (AB);
        \draw    (U) to[out=20,in=160] (AC);
        \draw    (V) to[out=20,in=160] (AD);
        \draw    (W) to[out=20,in=160] (AE);
        \draw    (X) to[out=20,in=160] (AF);
        
	\draw    (A) to[out=250,in=110] (Q);
        \draw    (B) to[out=250,in=110] (R);
        \draw    (E) to[out=250,in=110] (U);
        \draw    (F) to[out=250,in=110] (V);

        \draw    (C) to[out=290,in=70] (S);
        \draw    (D) to[out=290,in=70] (T);
        \draw    (G) to[out=290,in=70] (W);
        \draw    (H) to[out=290,in=70] (X);

        \draw    (I) to[out=250,in=110] (Y);
        \draw    (J) to[out=250,in=110] (Z);
        \draw    (M) to[out=250,in=110] (AC);
        \draw    (N) to[out=250,in=110] (AD);

        \draw    (K) to[out=290,in=70] (AA);
        \draw    (L) to[out=290,in=70] (AB);
        \draw    (O) to[out=290,in=70] (AE);
        \draw    (P) to[out=290,in=70] (AF);

	\draw (A) -- (B);
	\draw (A) -- (C);
	\draw (A) -- (F);
	\draw (B) -- (D);
	\draw (B) -- (E);
	\draw (C) -- (D);
	\draw (C) -- (G);
	\draw (D) -- (H);
	\draw (E) -- (F);
	\draw (E) -- (H);
	\draw (F) -- (G);
	\draw (G) -- (H);

	\draw (I) -- (J);
	\draw (I) -- (K);
	\draw (I) -- (N);
	\draw (J) -- (L);
	\draw (J) -- (M);
	\draw (K) -- (L);
	\draw (K) -- (O);
	\draw (L) -- (P);
	\draw (M) -- (N);
	\draw (M) -- (P);
	\draw (N) -- (O);
	\draw (O) -- (P);

	\draw (Q) -- (R);
	\draw (Q) -- (S);
	\draw (Q) -- (V);
	\draw (R) -- (T);
	\draw (R) -- (U);
	\draw (S) -- (T);
	\draw (S) -- (X);
	\draw (T) -- (W);
	\draw (U) -- (V);
	\draw (U) -- (W);
	\draw (V) -- (X);
	\draw (W) -- (X);

	\draw (Y) -- (Z);
	\draw (Y) -- (AA);
	\draw (Y) -- (AD);
	\draw (Z) -- (AB);
	\draw (Z) -- (AC);
	\draw (AA) -- (AB);
	\draw (AA) -- (AF);
	\draw (AB) -- (AE);
	\draw (AC) -- (AD);
	\draw (AC) -- (AE);
	\draw (AD) -- (AF);
	\draw (AE) -- (AF);

\end{tikzpicture}
\caption{$Q_5$, with a blue set that we think may be a $4$-leaky forcing set.}\label{fig:cube5}
\end{figure}
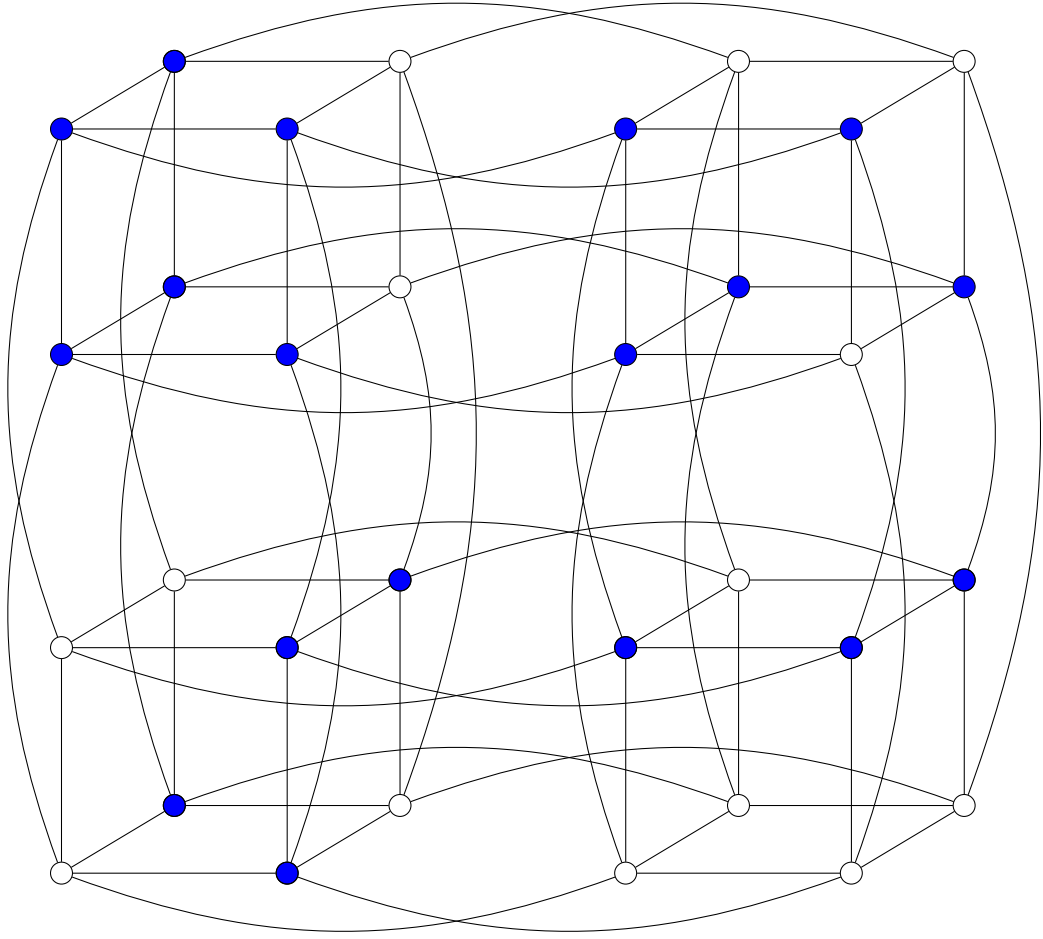

%\section{Ideas}
%lower the upper bound for $P_n \times P_m$.

%\begin{figure}
%\begin{tikzpicture}

%\def\numberlist{{3,0,1,0,2,3,0,2,1,404,2,1,0,6,7,1,1,3,4,5,1,3,2,1,0}} 
%\def\maxX{4}
%\foreach \x [count=\n] in {0,...,\maxX}{
 %   \foreach \y in {0,...,4}{
  %      \pgfmathtruncatemacro\tmp{\n+(\maxX+1)*\y}
   %     \pgfmathtruncatemacro\num{\numberlist[\tmp-1]}
    %    \draw[line width=.5pt] (\x,0) -- (\x,4);
     %   \draw[line width=.5pt] (0,\y) -- (\maxX,\y);
      %  \node[dot={45}{\num}] at (\x,\y) {};
 %   }
%}

%\end{tikzpicture}
%\caption{An example of drawing a grid with a for loop.}
%\label{fig:tightexample}
%\end{figure}

\section*{Acknowledgements}
G.W. was supported by University of Tennessee Knoxville Faculty Research Assistant Funding.

\bibliographystyle{abbrv}
\bibliography{LeakyForcing}

\begin{thebibliography}{10}

\bibitem{abbas2023resilient}
W.~Abbas.
\newblock Resilient strong structural controllability in networks using leaky forcing in graphs.
\newblock In {\em 2023 American Control Conference (ACC)}, pages 1339--1344. IEEE, 2023.

\bibitem{ahmad2023graph}
O.~U. Ahmad, M.~Shabbir, W.~Abbas, and X.~Koutsoukos.
\newblock A graph machine learning framework to compute zero forcing sets in graphs.
\newblock {\em IEEE Transactions on Network Science and Engineering}, 2023.

\bibitem{alameda2022leaky}
J.~S. Alameda, J.~Kritschgau, N.~Warnberg, and M.~Young.
\newblock On leaky forcing and resilience.
\newblock {\em Discrete Applied Mathematics}, 306:32--45, 2022.

\bibitem{alameda2020generalizations}
J.~S. Alameda, J.~Kritschgau, and M.~Young.
\newblock Generalizations of leaky forcing.
\newblock {\em arXiv preprint arXiv:2009.07073}, 2020.

\bibitem{benson2018zero}
K.~F. Benson, D.~Ferrero, M.~Flagg, V.~Furst, L.~Hogben, V.~Vasilevska, and B.~Wissman.
\newblock Zero forcing and power domination for graph products.
\newblock {\em Australasian Journal of Combinatorics}, 70(2):221, 2018.

\bibitem{brevsar2017grundy}
B.~Bre{\v{s}}ar, C.~Bujt{\'a}s, T.~Gologranc, S.~Klav{\v{z}}ar, G.~Ko{\v{s}}mrlj, B.~Patk{\'o}s, Z.~Tuza, and M.~Vizer.
\newblock Grundy dominating sequences and zero forcing sets.
\newblock {\em Discrete Optimization}, 26:66--77, 2017.

\bibitem{brevsar2024bootstrap}
B.~Bre{\v{s}}ar, J.~Hed{\v{z}}et, and R.~Herrman.
\newblock Bootstrap percolation and ${P}_3 $-hull number in direct products of graphs.
\newblock {\em arXiv preprint arXiv:2403.10957}, 2024.

\bibitem{burgarth2007full}
D.~Burgarth and V.~Giovannetti.
\newblock Full control by locally induced relaxation.
\newblock {\em Physical review letters}, 99(10):100501, 2007.

\bibitem{butler2013throttling}
S.~Butler and M.~Young.
\newblock Throttling zero forcing propagation speed on graphs.
\newblock {\em Australas. J. Combin}, 57:65--71, 2013.

\bibitem{cameron2024approximation}
B.~Cameron, J.~Janssen, R.~Matthew, and Z.~Zhang.
\newblock An approximation algorithm for zero forcing.
\newblock {\em arXiv preprint arXiv:2402.08866}, 2024.

\bibitem{cameron2023forts}
T.~R. Cameron, L.~Hogben, F.~H. Kenter, S.~A. Mojallal, and H.~Schuerger.
\newblock Forts,(fractional) zero forcing, and cartesian products of graphs.
\newblock {\em arXiv preprint arXiv:2310.17904}, 2023.

\bibitem{davila2018lower}
R.~Davila, T.~Kalinowski, and S.~Stephen.
\newblock A lower bound on the zero forcing number.
\newblock {\em Discrete Applied Mathematics}, 250:363--367, 2018.

\bibitem{dillman2019leaky}
S.~Dillman and F.~Kenter.
\newblock Leaky forcing: a new variation of zero forcing.
\newblock {\em arXiv preprint arXiv:1910.00168}, 2019.

\bibitem{elias2023leaky}
O.~Elias, I.~Farish, E.~King, J.~Kyei, and R.~Moruzzi~Jr.
\newblock Leaky positive semidefinite forcing on graphs.
\newblock {\em arXiv preprint arXiv:2312.10154}, 2023.

\bibitem{gao2015bootstrap}
J.~Gao, T.~Zhou, and Y.~Hu.
\newblock Bootstrap percolation on spatial networks.
\newblock {\em Scientific reports}, 5(1):14662, 2015.

\bibitem{gentner2018some}
M.~Gentner and D.~Rautenbach.
\newblock Some bounds on the zero forcing number of a graph.
\newblock {\em Discrete Applied Mathematics}, 236:203--213, 2018.

\bibitem{gravner2017bootstrap}
J.~Gravner and D.~Sivakoff.
\newblock Bootstrap percolation on products of cycles and complete graphs.
\newblock 2017.

\bibitem{hedvzet20233}
J.~Hed{\v{z}}et and M.~A. Henning.
\newblock $3 $-neighbor bootstrap percolation on grids.
\newblock {\em arXiv preprint arXiv:2307.14033}, 2023.

\bibitem{herrman2022d}
R.~Herrman.
\newblock The $(d- 2)$-leaky forcing number of ${Q}_d$ and $\ell$-leaky forcing number of ${GP} (n, 1)$.
\newblock {\em Discrete Optimization}, 46:100744, 2022.

\bibitem{hogben2012propagation}
L.~Hogben, M.~Huynh, N.~Kingsley, S.~Meyer, S.~Walker, and M.~Young.
\newblock Propagation time for zero forcing on a graph.
\newblock {\em Discrete Applied Mathematics}, 160(13-14):1994--2005, 2012.

\bibitem{huang2010minimum}
L.-H. Huang, G.~J. Chang, and H.-G. Yeh.
\newblock On minimum rank and zero forcing sets of a graph.
\newblock {\em Linear Algebra and its Applications}, 432(11):2961--2973, 2010.

\bibitem{karst2020blocking}
N.~Karst, X.~Shen, D.~S. Troxell, and M.~Vu.
\newblock Blocking zero forcing processes in cartesian products of graphs.
\newblock {\em Discrete Applied Mathematics}, 285:380--396, 2020.

\bibitem{leclair2024zero}
H.~LeClair, T.~Spilde, S.~Anderson, and B.~Kroschel.
\newblock Zero forcing of generalized hierarchical products of graphs.
\newblock {\em arXiv preprint arXiv:2410.17440}, 2024.

\bibitem{monshizadeh2014zero}
N.~Monshizadeh, S.~Zhang, and M.~K. Camlibel.
\newblock Zero forcing sets and controllability of dynamical systems defined on graphs.
\newblock {\em IEEE Transactions on Automatic Control}, 59(9):2562--2567, 2014.

\bibitem{mousavi2018null}
S.~S. Mousavi, A.~Chapman, M.~Haeri, and M.~Mesbahi.
\newblock Null space strong structural controllability via skew zero forcing sets.
\newblock In {\em 2018 European Control Conference (ECC)}, pages 1845--1850. IEEE, 2018.

\bibitem{trefois2015zero}
M.~Trefois and J.-C. Delvenne.
\newblock Zero forcing number, constrained matchings and strong structural controllability.
\newblock {\em Linear Algebra and its Applications}, 484:199--218, 2015.

\end{thebibliography}

\end{document}